\documentclass[10pt,oneside,reqno]{amsart}
\usepackage{amsmath,amssymb}
\usepackage{amsthm} 

\newtheorem{thm}{Theorem}[section]
\newtheorem{cor}[thm]{Corollary}
\newtheorem{lemma}[thm]{Lemma}

\theoremstyle{definition}

\theoremstyle{remark}

\numberwithin{equation}{section}

\def\bes{\begin{equation*}}
\def\be{\begin{equation}}
\def\ee{\end{equation}}
\def\ees{\end{equation*}}

\allowdisplaybreaks

\begin{document}

\title[Asymptotic expansions of exponentials of digamma function]{Asymptotic expansions of exponentials of digamma function and identity for Bernoulli polynomials}

\author{Neven Elezovi\'c}

\address{Neven Elezovi\'c, Faculty of Electrical Engineering and Computing, University of
Zagreb, Unska 3, 10000 Zagreb, Croatia}
\email{neven.elez\@fer.hr}

\begin{abstract}
	The asymptotic expansion of digamma function is a starting point for the derivation of approximants for harmonic sums or Euler-Mascheroni constant. It is usual to derive such approximations as values of logarithmic function, which leads to the expansion of the exponentials of digamma function. In this paper
	the asymptotic expansion of the function $\exp(p\psi(x+t))$ is derived and analyzed in details, especially for integer values of parameter $p$. The behavior for integer values of $p$ is proved and as a consequence a new identity for Bernoulli polynomials. The obtained formulas are used to improve know inequalities for Euler's constant and harmonic numbers.
\end{abstract}



\thanks{Date of submission: \today}

\maketitle

\section{Introduction}

        Let $\gamma$ denote Euler's constant, $H_n$ the harmonic number --- the partial sum of harmonic series:
        \bes
                H_n=1+\frac12+\dots+\frac1n,
        \ees
        and $\psi$ digamma (psi) function.

        Harmonic numbers, logarithm, Euler's constant and digamma function are connected through well known relations, the main one is
        \be
                \psi(n+1)=H_n-\gamma.
        \label{psiHn}
        \ee
        Various approximations of digamma function are used in this relation and interpreted as approximation for the sequence $(H_n)$ or $\gamma$.
        In the recent paper \cite{ChMo-2012-1} the constants $a$, $b$, $c$, $d$ in the expression
        \be
                w_n=H_n-\ln\biggl(n+a+\frac bx+\frac c{n^2}+\frac d{n^3}\biggr)
        \ee
        are calculated such that this sequence is a good approximation to $\gamma$. 
        The authors obtained four term expansion 
        \bes
                e^{\psi(x+1)}=x+\frac12+\frac1{24x}-\frac1{48x^2}+\frac{23}{5760x^3}+O(x^{-4})
        \ees
        which can be written conventionally as
        \be
                \psi(x+1)=\log\biggl(x+\frac{1}{2}+\frac{1}{24x}
                -\frac{1}{48x^2}+\frac{23}{5760x^3}+O(x^{-4})\biggr).
        \ee
        The calculation was tedious with no clue how one can obtain general term. In fact, a question about generalization of this procedure is posed.

        This question is answered in the paper \cite{Yang-2012-1}, but the algorithm proposed there is also complicated. It relies on the connection between logarithmic function and Bell polynomials, but the algorithm for calculation of this polynomials is not easy. 

        In this paper we shall explain a simple and useful algorithm which covers more general case of approximations of this type. We will not restrict ourself to finding a numerical sequence. Instead, there is a natural choice of at least one parameter which can better explain obtained expansions and related inequalities.
        So, we shall consider the expansion of the function $e^{p\psi(x+t)}$, where $p$ and $t$ are arbitrary real numbers. The case when $p$ is a natural number is especially important and curious, and it will be explained in details.

        Replacing argument $x+1$ by $x+t$ leads to the better understanding of the involving expansions. The magnificent formula
        \bes
                \psi(x+t)\sim\log x+\sum_{n=0}^\infty (-1)^n\frac{B_n(t)}nx^{-n}
        \ees
        is a good example of the role of introducing a parameter in such expansions.

        The following lemma about functional transformations of asymptotic series has his origin in Euler's work. 
	 See e.g.\ \cite{G} for the explanation in the case of Taylor series, and \cite{CEV} for its use in the case of asymptotic series.

\begin{lemma}
        \label{potl}
        Let $s$ be a real number, $a_0=1$ and $g(x)$ a function with asymptotic expansion 
        \bes
                g(x)\sim \sum_{n=0}^{\infty}a_nx^{-n}.
        \ees
        Then for all real $p$ it holds
        \bes
                g(x)^{p}\sim \sum_{n=0}^{\infty}b_n(p)x^{-n},
        \ees
        where   
        \be
        \begin{aligned}
                b_0(p)&=1,\\
                b_n(p)&=\frac{1}{n}\sum_{k=1}^{n}[k(1+p)-n]a_k b_{n-k}(p).
        \end{aligned}
        \label{pot}
        \ee
\end{lemma}

\section{The main asymptotic approximations}

        The following theorem is proved in \cite{CEV}:
        
\begin{thm}
        The following asymptotic expansion is valid as $x\to\infty$:
        \be
                \psi(x+t)\sim \log\biggl(\sum_{n=0}^\infty S_n(t)x^{-n+1}\biggr)
        \label{Sn}
        \ee
        where $S_0=1$ and
        \be
                S_n(t)=\frac1n\sum_{k=1}^n (-1)^{k+1}B_k(t)
                S_{n-k}(t),\qquad n\ge1.
        \ee
\end{thm}

        The first few coefficients are:
        \begin{align*}
                S_0&=1,\\
                S_1&=t-\tfrac12,\\
                S_2&=\tfrac1{24},\\
                S_3&=-\tfrac1{24}t+\tfrac1{48},\\
                S_4&= \tfrac1{24}t^2-\tfrac1{24}t+\tfrac{23}{5760},\\
                S_5&=-\tfrac1{24}t^3+\tfrac1{16}t^2-\tfrac{23}{1920}t-\tfrac{17}{3840},\\
                S_6&=\tfrac1{24}t^4-\tfrac1{12}t^3+\tfrac{23}{960}t^2+\tfrac{17}{960}-\tfrac{10099}{2903040}.
        \end{align*}

        The collaps of the degree in the polynomial $S_2$ is explained and various relations about polynomials $(S_n)$ are proved.
        This theorem will be extended in our main theorem which covers more complex situation.

        From this list of the coefficients one can write the following two expansions, for $t=0$
        \be
                \psi(x)\sim\log\biggl(x+\frac{1}{2}+\frac{1}{24x}
                -\frac{1}{48x^2}+\frac{23}{5760x^3}-\frac{17}{3840x^4}
                -\frac{10099}{2903040x^5}+\dots\biggr).
        \label{t=0}
        \ee
        and for $t=\frac12$
        \be
                \psi(x+\tfrac12)\sim\log\biggl(x+\frac1{24x}-\frac{37}{5760x^3}
                +\frac{10313}{2903040x^5}-\frac{5509121}{1393459200x^7}+\dots\biggr)
        \ee
	These expansions are used in the paper \cite{BaCh-2012-1} as approximants for harmonic sums, but the calculations of coefficients are made by other more tedious methods which cannot be improved to obtain general term of these sequences.

        From \eqref{t=0} one can obtain the following expansions
        \bes
                e^{2\psi(x)}\sim x^2-x+\frac13-\frac1{90x^2}-\frac1{90x^3}
                -\frac1{567x^4}+\frac{43}{5670x^5}+\dots
        \ees
        and
        \bes
                e^{4\psi(x)}\sim x^4-2x^3+\frac{5x^2}3-\frac{2x}3+\frac4{455}+\frac{32}{2835x^2}+\frac{32}{2835x^3}+\dots
        \ees
        The fact that in both expansions the term with $x^{-1}$ is missing is not a coincidence. This is true for all even $p$, but this is by no means obvious. We shall prove this in the next section.

        Let us first derive asymptotic expansion of the function
        $e^{p\psi(x+t)}$ where $p$ is a real number. This expansion
        can be written in the form
        \bes
                p\psi(x+t)\sim\log\biggl[x^p\biggl(\sum_{n=0}^\infty S_n(t)x^{-n}\biggr)^p\biggr].
        \ees
        Therefore, it is enough to apply algorithm given in Lemma~\ref{potl} to the calculated polynomials $(S_n)$
        to obtain the following:

\begin{thm}
        Let $p$ be a real number.
        Function $e^{p\psi(x+t)}$ has the following asymptotic expansion
        \be
                e^{p\psi(x+t)}
                \sim x^p\sum_{n=0}^\infty G_n(p,t)x^{-n}
        \label{Gn-exp}
	\ee
        where $(G_n)$ are defined by $G_0=1$ and
        \be
                G_n:=\frac1n\sum_{k=1}^n[k(1+p)-n]S_n(t)G_{n-k}.
        \label{GnLn}
        \ee
\end{thm}

	Let us denote in the sequell
	\be
		G(p,t,x)=x^p\sum_{n=0}^\infty G_n(p,t)x^{-n}.
	\ee

        The polynomials $(G_n)$ are easy to calculate using any CAS as functions of both $t$ and $p$, but it is not easy to write it down. Polynomial $S_n$ is of degree $n-2$ for $n\ge 2$, but 
        polynomials $t\mapsto G_n(t,p)$ and $p\mapsto G_n(t,p)$ will be, in general, of degree $n$ --- this is obvious from recurrent relation \eqref{GnLn}.
        
        Here are the first few polynomials $(G_n)$:
        \begin{align*}
                G_0&=1,\\
                G_1&=\tfrac12p(1-2t),\\
                G_2&=\tfrac1{24}p(-2+3p+12t-12pt-12t^2+12pt^2),\\
                G_3&=\tfrac1{48}(-2+p)p(-1+2t)(p+4t-4pt-t^2+4pt^2),\\
                G_4&=\tfrac1{5760}p(15p^3-60p^2+20p+48)-\tfrac1{48}
                (p^2(p-2)(p-3)t\\
                &\qquad+\frac1{48}p(p-2)(p-3)(3p-2)t^2
                -\tfrac1{12}p(p-1)(p-2)(p-3)t^3\\
                &\qquad
                +\tfrac1{24}p(p-1)(p-2)(p-3)t^4.
        \end{align*}

        In order to highlight some interesting details, we will choose some concrete value of one of these parameters.
        For example, $p=2$ and $p=3$ gives the following sequence:
        \begin{alignat*}{2}
                (p=2)&&(p=3)&\\
                G_0&=1,         
                &G_0&=1,
                \\
                G_1&=2t-1,
                &G_1&=3t-\frac32,
                \\
                G_2&=t^2-t+\frac13,
                &G_2&=3t^2-3t+\frac78,
                \\
                G_3&=0,
                &G_3&=\frac1{16}(2t-1)(8t^2-8t+2)
                \\
                G_4&=-\frac1{90},
                &G_4&=-\frac9{640},
                \\
                G_5&=\frac{2t-1}{90},
                &G_5&=\frac{9(2t-1)}{1280},
                \\
                G_6&=-\frac{t^2}{30}+\frac t{30}-\frac1{567},
                \qquad
                &G_6&=\frac{201 t^2}{35840}+\frac{9t}{640}=\frac{9}{640}
        \end{alignat*}
        One can see the collapse of the polynomial $G_{p+1}$. The similar thing can be noticed for each positive integer $p$. For example, if $p=4$ then $G_5=0$.

        First, we shall deduce another, more natural recursion for polynomials $(G_n)$. It is based on the following result which is proved in \cite{BuEl-2011-1} and \cite[Theorem 3.1.]{BuEl-2012-1}.

        Let $F(x,s,t)$ be Wallis function, defined by
        \bes
                F(x,s,t)=\frac{\Gamma(x+t)}{\Gamma(x+s)}
        \ees
        Then, it holds
        \bes
                F(x,s,t)\sim x^{t-s}\left[\sum_{n=0}^\infty P_n(t,s)x^{-n}\right]^{\tfrac1m}
        \ees
        where $P_0=1$ and $(P_n)$ is defined by
        \bes
                P_n(t,s)=\frac mn\sum_{k=1}^n
                (-1)^{k+1}
                \frac{B_{k+1}(t)-B_{k+1}(s)}{k+1}P_{n-k}(t,s).
        \ees
	Now, let us choose $m=p/(t-s)$.
        Taking the limit $s\to t$, we obtain:
        \bes
                \lim_{s\to t}F(x,s,t)^{p/(t-s)}=e^{p\psi(x+t)}
        \ees
        From the other side, using property of Bernoulli polynomials,
        \bes
                p\lim_{s\to t}\frac{B_{k+1}(t)-B_{k+1}(s)}{(k+1)(t-s)}
                =p\frac{B'_{k+1}(t)}{k+1}=p B_k(t),
        \ees
        we obtain the following result.

\begin{thm}
        For any real $p\ne0$ function $e^{p\psi(x+t)}$ has the asymptotic expansion \eqref{Gn-exp} where $G_0=1$ and
        \be
                G_n=\frac{p}{n}\sum_{k=1}^n(-1)^{k+1}B_k(t)G_{n-k}.
        \label{GnBk}
	\ee
\end{thm}

%%%%%%%%%%%%%%%%%%%%%%%%%%%%%%%%%%%%%%%%%%%%%%%%%%%%%%%%%%%%
%%%%%%%%%%%%%%%%%%%%%%%%%%%%%%%%%%%%%%%%%%%%%%%%%%%%%%%%%%%%
%%%%%%%%%%%%%%%%%%%%%%%%%%%%%%%%%%%%%%%%%%%%%%%%%%%%%%%%%%%%
\section{Integer values of $p$} 

        Now we can prove the main results.

\begin{thm}
        Polynomials $(G_n)$ satisfy the following identity
        \be
                G_n(p,s+t)=\sum_{k=0}^n \binom{p-n+k}{k}
                G_{n-k}(p,s)t^{k}.
        \label{t+s}
	\ee

	1) If $p$ is not positive integer, then $G_n(p,t)$ is polynomial of degree $n$ in both variables.

        Let us suppose that $p$ is positive integer. Then

        2) For $n\le p$, $t\mapsto G_n$ is polynomial of degree $n$.

        3) For $n\ge p+1$, $t\mapsto G_n$ is polynomial of degree $\le n-p-1$.
\end{thm}

\begin{proof}
	Let $p$, $t$ and $s$ be arbitrary real numbers.
	The expansion for the function $G(p,s+t,x)$ can be written in two different ways:
	\begin{align*}
		G(p,s+t,x)&=\sum_{n=0}^\infty G_n(p,s+t)x^{-n+p}\\
		&=\sum_{n=0}^\infty G_n(p,s)(x+t)^{-n+p}\\
		&=\sum_{n=0}^\infty G_n(p,s)\sum_{k=0}^{\infty}
		\binom{-n+p}{k}t^kx^{-n+p-k}\\
		&=\sum_{n=0}^\infty
		\biggl[\sum_{k=0}^n
		\binom{p-n+k}kG_{n-k}(p,s)t^k
		\biggr]x^{-n+p}
	\end{align*}
	which proves \eqref{t+s}.

	The assertions 1) and 2) are obvious from this explicit formula.
	Let us prove 3). Suppose $n\ge p+1$. Then, from \eqref{t+s} 
	we can write
	\bes
                G_n(p,t)=\sum_{k=0}^n (-1)^k \binom{n-p+1}{k}
                G_{n-k}(p,0)t^{k}.
	\ees

	If $n=p+1$, the sum is reduced to the first term:
	\bes
		G_{p+1}(p,t)=G_{p+1}(p,0),
	\ees
	so, polynomial $G_{p+1}(p,t)$ is of degree $0$, or it is equal to zero identically.

	For $n>p+1$,
	\bes
                G_n(p,t)=\sum_{k=0}^{n-p-1}(-1)^k \binom{n-p+1}{k}
                G_{n-k}(p,0)t^{k}
	\ees
	and this is polynomial of degree at most $n-p-1$. This degree can be reduced if $G_{p+1}(p,0)=0$.
\end{proof}

	It will be shown that for each even $p$, it holds $G_{p+1}(p,0)=0$, so, the asymptotic expansion of $e^{p\psi(x)}$ does not contain the member with power $x^{-1}$. This requires another treatment. 

\begin{thm}
	\label{evenp}
        If $p$ is even natural number, then $G_{p+1}\equiv0$ and $t\mapsto G_{p+2+k}$ is polynomial of degree $k$. 
\end{thm}

\begin{proof}
	Let us treat $p$ as real valued variable. Then
	the polynomial $p\mapsto G_n(p,t)$ is of degree $n$. Let us denote its coefficients by:
        \be
                G_n(p,t)=\sum_{k=0}^n G_{n,k}(p)t^n.
        \label{Gn}
        \ee
        Since for the function $G(p,t,x)=e^{p\psi(x+t)}$ it holds
        $
                \dfrac{\partial G}{\partial x}=
                \dfrac{\partial G}{\partial t},
        $
        we have
        \bes
        \sum_{n=1}^\infty\frac{\partial G_n}{\partial t}
        x^{-n+p}=\sum_{n=1}^\infty G_{n-1}(-n+p+1)x^{-n+p},
        \ees
        therefore
        \bes
                \frac{\partial G_n(p,t)}{\partial t}
                =(p+1-n)G_{n-1}(p,t).
        \ees
        Now, from the \eqref{Gn} it follows
        \bes
                G_{n,k}=\frac{p+1-n}{k}G_{n-1,k-1}.
        \ees
        Therefore,
        \be
                G_{n,k}=\binom{p-n+k}k G_{n-k,0}.
        \label{Gnk}
	\ee

        Let us arrange coefficients in the following table:
        \bes
        \begin{array}{cccccccc}
        G_{0,0}\\
        G_{1,0}&G_{1,1}\\
        G_{2,0}&G_{2,1}&G_{2,2}\\
        G_{3,0}&G_{3,1}&G_{3,2}&G_{3,3}\\
        \vdots\\
        G_{n-1,0}&G_{n-1,1}&G_{n-1,2}&\cdots&G_{n-1,n-1}\\
        G_{n,0}&G_{n,1}&G_{n,2}&\cdots&G_{n,n-1}&G_{n,n}\\
        \end{array}
        \ees
        If some of the coefficients $G_{n,k}$ in this table are equal to zero for some particular value of $p$, then the same is true for all coefficients $G_{n+1,k+1}$, $G_{n+2,k+2}$,\dots.

	To finish proof of the theorem, we should prove that 
        \bes
                G_{p+1,k}=0,\qquad \text{for all }k
        \ees
        if $p$ is an even number. 
	
	First, note that for $k\ge1$ we read from \eqref{Gnk} that $G_{n,k}$ is divisible by $(p-n+1)\cdots
	(p-n+k)$. Therefore, for each integer value od $p$, all coefficients $G_{p+1,k}$, $k\ge1$ are equal to zero. Then, of course, $G_{p+1}(p,t)\equiv G_{p+1}(p,0)$, but, also, the degree of all subsequent polynomials are reduced by $p$.

	To end this proof, we need the following

\begin{lemma}
        It holds
        \be
                G_n(p,1)=(-1)^n G_n(p,0).
        \ee
\end{lemma}

        In other words,
        \bes
                \sum_{k=0}^n G_{n,k}(p,0)=(-1)^nG_{n,0}(p,0)
        \ees
        From here, if $p$ is an even number
        \bes
                G_{p+1,0}=-\frac12\sum_{k=1}^n G_{p+1,k}(p+1,0)=0.
        \ees
	and the theorem is proved.

        Finally, let us prove previous lemma using induction. We have, using \eqref{GnBk}
        \begin{gather*}
                G_0(p,1)=1=G_0(p,0),\\
                G_1(p,1)=B_1(1)G_0(p,1)=-B_1(0)G_0(p,0)=-G_1(p,0).
        \end{gather*}
        Let us suppose that
        \bes
                G_k(p,1)=(-1)^k G_k(p,0)
        \ees
        is satisfied for all $k=0,1,\dots,n-1$. Then
        \begin{align*}
                G_n(p,1)&=\frac1n\sum_{k=1}^n(-1)^{k+1}B_k(1)G_{n-k}(p,1)\\
                &=\frac1n\sum_{k=1}^n(-1)^{k+1}(-1)^kB_k(0)
                (-1)^{n-k}G_k(p,0)\\
                &=(-1)^n\frac1n\sum_{k=1}^n(-1)^{k+1}B_k(0)G_{n-k}(p,0)\\
                &=(-1)^n G_n(p,0).
        \end{align*}
        This proves the lemma.
\end{proof}

%%%%%%%%%%%%%%%%%%%%%%%%%%%%%%%%%%%%%%%%%%%%%%%%%%%%%%%%%
%%%%%%%%%%%%%%%%%%%%%%%%%%%%%%%%%%%%%%%%%%%%%%%%%%%%%%%%%
%%%%%%%%%%%%%%%%%%%%%%%%%%%%%%%%%%%%%%%%%%%%%%%%%%%%%%%%%

\section{Fixed values of $t$}

        Interesting value is $t=1$, wherefrom one obtains
        \begin{align*}
                G_0&=1,\\
                G_1&=\frac p2,\\
                G_2&=\frac1{24}p(3p-2),\\
                G_3&=\frac1{48}p^2(p-2),\\
                G_4&=\frac1{5760}p(15p^3-60p^2+20p+48),\\
                G_5&=\frac1{11520}p^2(p-4)(3p^2-8p-12).
        \end{align*}
        For $t=\frac12$ this sequence is reduced to even members, i.e. $G_{2n+1}=0$ and:
        \begin{align*}
                G_0&=1,\\
                G_2&=\frac p{24},\\
                G_4&=\frac{p(5p-42)}{5760}),\\
                G_6&=\frac{p(35p^2-882p+11160}{2903040}.
        \end{align*}
        Here $p\mapsto G_{2n}(p,t)$ is a polynomial of degree only $p$, which shows that the interplay between $p$ and $t$ is not so obvious.

\begin{thm}
        For $t=\frac12$, polynomials $G_{2n+1}$ vanishes, 
	$G_{2n}$ has the degree $n$ and can be calculated from
	\be
	G_{2n}=-\frac{p}{2n}\sum_{k=1}^n(1-2^{-2k})B_{2k}G_{2n-2k}.
	\label{t=1/2}
	\ee
\end{thm}

\begin{proof}
	We will use the expression (\ref{GnBk}).
        It is enough to use the property of Bernoulli polynomials
	\bes
		B_{k}(\tfrac12)=(1-2^{-2k})B_k.
	\ees
	Hence, for all odd $k$ we have $B_k=0$, so, using 
	\eqref{GnBk}
	\begin{align*}
		G_n&=\frac pn\sum_{k=1}^n B_k(\tfrac12)G_{n-k}\\
		&=\frac pn\sum_{k=1}^{\lfloor n/2\rfloor}
		(-1)^{2k+1}(1-2^{2k})B_{2k}G_{n-2k}
	\end{align*}
	Hence, for odd $n$ there is no summands on the right, and for even $n$ one obtains \eqref{t=1/2}. From there, it is evident that $G_{2n}$ is polynomial of degree $n$.
\end{proof}

%%%%%%%%%%%%%%%%%%%%%%%%%%%%%%%%%%%%%%%%%%%%%%%%%%%%%%%%%%%%%
%%%%%%%%%%%%%%%%%%%%%%%%%%%%%%%%%%%%%%%%%%%%%%%%%%%%%%%%%%%%%
%%%%%%%%%%%%%%%%%%%%%%%%%%%%%%%%%%%%%%%%%%%%%%%%%%%%%%%%%%%%%
%%%%%%%%%%%%%%%%%%%%%%%%%%%%%%%%%%%%%%%%%%%%%%%%%%%%%%%%%%%%%
\section{An identity for Bernoulli polynomials}

	Polynomials $(G_n)$ have explicit formula through Bernoulli polynomials, which is not friendly for its calculations. It is similar to the expression from \cite{Yang-2012-1}:

\begin{thm}
	The following explicit formula for coefficients $G_n$ is valid:
	\be
	(-1)^nG_n(p,t)=\sum_{r=1}^n\frac{(-p)^r}{r!}
	\sum_{\substack{
	k_1+\dots+k_r=n\\k_i\ge1
	}}
	\frac{B_{k_1}(t)\cdots B_{k_r}(t)}{k_1\cdots k_r}.
	\label{Bn-exp}
	\ee
\end{thm}

\begin{proof}
	We have 
	\begin{align*}
		G(p,t,x)&=x^p\exp\biggl(-p\sum_{k=0}^\infty
		\frac{B_k(1-t)}k\,x^{-k}\biggr)\\
		&=x^p\biggl(1+\sum_{r=1}^\infty
		\frac{(-p)^r}{r!}
		\biggl[\sum_{k=0}^\infty
		\frac{B_k(1-t)}{k}x^{-k}\biggr]^r\biggr)
	\end{align*}
	Hence
	\bes
	G_n(p,t)=\sum_{r=1}^n\frac{(-p)^r}{r!}
	\sum_{\substack{
	k_1+\dots+k_r=n\\k_i\ge1}}
	\frac{B_{k_1}(1-t)\cdots B_{k_r}(1-t)}{k_1\cdots k_r}.
	\ees
	and from here one obtains \eqref{Bn-exp}.
\end{proof}

	From the derivation of this expression it is obvious that
	indexes $k_1,k_2,\dots,k_r$ are not ordered, so, the combinations which lead to the same monomial should be counted carefully.

	Taking into account the result from Theorem~\ref{evenp} that
	left side vanishes for $n=p+1$, where $p$ is even, we can write the following identity for Bernoulli polynomials.

\begin{cor}
	For each natural number $n$, we have:
	\be
		\sum_{r=1}^{2n+1}
		\frac{(-2n)^r}{r!}
		\sum_{\substack{
	k_1+\dots+k_r=2n+1\\k_i\ge1}}
	\frac{B_{k_1}(t)\cdots B_{k_r}(t)}{k_1\cdots k_r}=0.
	\ee
\end{cor}

	For example, for $n=1$ and $n=2$ the following identities are fullfiled:
	\begin{gather*}
		\tfrac23B_3(t)+2B_1(t)B_2(t)-\tfrac43B_1(t)^3=0,\\
		-\tfrac45B_5(t)+
		+ 4 B_1(t)B_4(t) + \tfrac83 B_2(t)B_3(t) - 
		\tfrac{32}3 B_1(t)^2B_3(t)\qquad\qquad\\
		\qquad\qquad- 
		 8 B_1(t)B_2(t)^2 + \frac{64}3 B_1(t)^3 B_2(t) 
		 - \frac{128}{15}B_1(t)^5=0.
	\end{gather*}

\end{document}